\documentclass{amsart}

\usepackage{amsmath}
\usepackage{amsthm}
\usepackage{graphicx}
\usepackage{color}
\usepackage{amsfonts}
\usepackage{amssymb}
\usepackage{mathrsfs}
\usepackage{amscd}
\usepackage{url,enumitem}

    \usepackage{tikz,pgfplots}
    \usetikzlibrary{shapes,patterns,calc} 
\usepgfplotslibrary{fillbetween}
\definecolor{ogreen}{RGB}{85,107,47} % a darker green
\definecolor{dred}{RGB}{139,0,0}

\def\foo{1}
\ifx\foo\undefined
{}
\else
\usepackage[pagebackref,colorlinks,linkcolor=dred,
citecolor=ogreen,urlcolor=black,hypertexnames=true]{hyperref}
\usepackage[margin=2.5cm]{geometry}
\usepackage[fontsize=12]{scrextend}
\fi
\newcommand{\re}{\mathbb{R}}
\newcommand{\R}{\mathbb{R}}

\newcommand{\Z}{\mathbb{Z}}

\newcommand{\N}{\mathbb{N}}

\newcommand{\diag}{\mbox{diag}}

\newcommand{\half}{\frac{1}{2}}

\newcommand{\nn}{\nonumber}
\newcommand{\eps}{\epsilon}

\newcommand{\dt}{\delta}

\def\af{\alpha}
\def\bt{\beta}
\def\la{\lambda}
\def\gm{\gamma}
\def\Gm{\Gamma}

\newcommand{\sig}{\sigma}

\newcommand{\reff}[1]{(\ref{#1})}

\newcommand{\mbS}{\mathcal{S}}

\newcommand{\mc}[1]{\mathcal{#1}}

\DeclareMathOperator{\QM}{QM}

\newcommand{\x}{{\tt x}}
\newcommand{\w}{{\tt w}}
\newcommand{\y}{{\tt y}}
\newcommand{\z}{{\tt z}}
\newcommand\cx{[\x]}

\newcommand\rx{\R\cx}

\newcommand{\bdes}{\begin{description}}
\newcommand{\edes}{\end{description}}

\newcommand{\bal}{\begin{align}}
\newcommand{\eal}{\end{align}}

\newcommand{\bnum}{\begin{enumerate}}
\newcommand{\enum}{\end{enumerate}}

\newcommand{\bit}{\begin{itemize}}
\newcommand{\eit}{\end{itemize}}

\newcommand{\bea}{\begin{eqnarray}}
\newcommand{\eea}{\end{eqnarray}}
\newcommand{\be}{\begin{equation}}
\newcommand{\ee}{\end{equation}}

\newcommand{\baray}{\begin{array}}
\newcommand{\earay}{\end{array}}

\newcommand{\bsry}{\begin{subarray}}
\newcommand{\esry}{\end{subarray}}

\newcommand{\bca}{\begin{cases}}
\newcommand{\eca}{\end{cases}}

\newcommand{\bcen}{\begin{center}}
\newcommand{\ecen}{\end{center}}

\newcommand{\bbm}{\begin{bmatrix}}
\newcommand{\ebm}{\end{bmatrix}}

\newcommand{\bmx}{\begin{matrix}}
\newcommand{\emx}{\end{matrix}}

\newcommand{\bpm}{\begin{pmatrix}}
\newcommand{\epm}{\end{pmatrix}}

\newcommand{\btab}{\begin{tabular}}
\newcommand{\etab}{\end{tabular}}

\newtheorem{theorem}{Theorem}[section]

\newtheorem{prop}[theorem]{Proposition}

\newtheorem{cor}[theorem]{Corollary}
\newtheorem{corollary}[theorem]{Corollary}

\theoremstyle{definition}

\newtheorem{exm}[theorem]{Example}

\newtheorem{ques}[theorem]{Question}

\numberwithin{equation}{section}

\begin{document}

\title[A Matrix Positivstellensatz with lifting polynomials]
{A Matrix Positivstellensatz with\\ lifting polynomials} 
%and containment of spectrahedrops}

\author{Igor Klep}
\address{Igor Klep, Department of Mathematics,
The University of Auckland, New Zealand}
\email{igor.klep@auckland.ac.nz}

\author{Jiawang Nie}
\address{Jiawang Nie, Department of Mathematics,
University of California San Diego,
9500 Gilman Drive, La Jolla, CA, USA, 92093.}
\email{njw@math.ucsd.edu}


%
% Usually <= 5 keywords are allowed
%

\begin{abstract}
\iffalse
%%%%%% OLD
This article studies how to check the
containment between projections of two
sets defined by polynomial matrix inequalities.
To certify the containment,
we propose a new matrix Positivstellensatz
that uses lifting polynomials.
Under the classical archimedean condition
and some mild natural assumptions, we prove that
such a containment holds if and only if
the proposed matrix Positivstellensatz 
%using lifting polynomials
is satisfied. The corresponding %Positivstellensatz 
certificate
can be searched for by solving a semidefinite program.
An important application 
%of this Positivstellensatz 
is to
certify when a spectrahedrop (i.e., the projection of a spectrahedron)
is contained in another one.
%%%%%%%%%%%%%%%%%%%%%END old
\fi
%%%%%%%%%%%%%%%%%%%%%NEW
Given the projections of two semialgebraic sets defined by polynomial matrix inequalities,
it is in general difficult to determine whether one
is contained in the other.
To address this issue
we propose a new matrix Positivstellensatz
that uses lifting polynomials.
Under the classical archimedean condition
and some mild natural assumptions, we prove that
such a containment holds if and only if
the proposed matrix Positivstellensatz 
is satisfied. The corresponding certificate
can be searched for by solving a semidefinite program.
An important application 
is to
certify when a spectrahedrop (i.e., the projection of a spectrahedron)
is contained in another one.
\end{abstract}

\maketitle

\section{Introduction}

A basic question of fundamental importance in convex geometry and
optimization is to determine
whether or not containment holds between two given convex sets.
The simplest convex sets are polyhedra, defined by a finite set
of scalar linear inequalities.
Containment problems for polyhedra have been studied extensively  and are
well understood \cite{FO,GK}.
%
%Freund and Orlin studied containment problems of balls in balls~\cite{FO}.
%
Another class of thoroughly studied convex sets are spectrahedra.
They arise as feasible sets of semidefinite programs \cite{deK,SDPbk}
and are defined by linear matrix inequalities (LMIs).
Denote by $\mc{S}^k$ the space of all $k\times k$ real symmetric matrices.
A tuple
$
A := (A_0, A_1, \ldots, A_n) \in (\mc{S}^k)^{n+1}
$
gives rise to the linear pencil
\[
A(\x) := A_0 + \x_1 A_1 + \cdots + \x_n A_n,
\]
in the variable $\x:=(\x_1, \ldots, \x_n)$.
% \[
% B(x) := B_0 + x_1 B_1  + \cdots + x_nB_n
% \]
% where each $A_i, B_i$ is a symmetric matrix,
% give rise to the spectrahedra
% \[
% S_A = \{ x \in \re^n: \, A(x) \succeq 0 \}, \quad
% S_B = \{ x \in \re^n: \, B(x) \succeq 0 \}.
% \]
It determines the \emph{spectrahedron}
(i.e., a set that is defined by a linear matrix inequality)
\[
S_A := \{ x \in \re^n: \, A(x) \succeq 0 \}.
\]
(Here, $C\succeq0$ means
the matrix $C$ is positive semidefinite.
Similarly, we use $C\succ0$ to express 
that  $C$ is positive definite.)

Another important containment question is the matrix cube problem
by Ben-Tal \& Nemirovski~\cite{BTN,Nem}.
It asks for the largest hypercube contained
in a given spectrahedron. The problem is known to be NP-hard.
Numerous problems of robust control,
such as Lyapunov stability analysis for uncertain dynamical systems,
are special cases of the matrix cube problem.
This is also the case for maximizing a positive definite quadratic form
over the unit cube, one of the fundamental problems in combinatorial optimization.

More generally, given another tuple
\(
B := (B_0, B_1, \ldots, B_n) \in (\mc{S}^t)^{n+1},
\)
where $t$ might be different from $k$,
one is interested in a certificate for the containment
\be\label{eq:S_A in S_B}
S_A \subseteq S_B.
\ee
Clearly, if there exist matrices
$V_i$ ($i = 0, \ldots, \ell)$ such that
\be \label{B=VAV+W}
B(\x) =
V_0^TV_0
+
\sum_{i=1}^\ell  V_i^T A(\x) V_i,
\ee
then $S_A \subseteq S_B$. If $S_A$ has nonempty interior
(this is the case if $A_0=I_d$,
the $d \times d$ identity matrix), then
\reff{B=VAV+W} holds if and only if the matricial relaxation of
$S_A$ is contained in the matricial relaxation of
$S_B$~\cite{HKM12,HKM13}.
When $B(\x)$ is the normal form of an ellipsoid or polytope,
the certificate \reff{B=VAV+W} is necessary and sufficient
for $S_A \subseteq S_B$, as shown by Kellner, Theobald and Trabandt \cite{KTT13}.
%
%An alternative method for checking $S_A \subseteq S_B$
%is given by the same authors in \cite{KTT15}.
%
More general spectrahedral containment
is also addressed by the same authors in \cite{KTT15}.

In general, the certificate \reff{B=VAV+W} is
sufficient but not necessary for ensuring $S_A \subseteq S_B$.
A more general certificate than \reff{B=VAV+W} is
\be \label{B(x)=VA(x)V+W(x)}
B(\x) = V_0(\x)^TV_0(\x)+ \sum_{i=1}^\ell  V_i(\x)^T A(\x) V_i(\x) ,
\ee
for matrix polynomials $V_0(\x), V_1(\x), \ldots, V_\ell(\x)$.
To guarantee \reff{B(x)=VA(x)V+W(x)},
we typically need that $S_A$ is bounded
and $B(\x) \succ 0$  on $S_A$.
%(the inequality $\succ$ means
%the matrix is positive definite).
The boundedness of $S_A$ is equivalent to archimedeanness
of the quadratic module associated to the linear pencil $A(\x)$;
see \cite{KSmor}.
Hence if $S_A$ is bounded and $B(x) \succ 0$ on $S_A$,
then $B(x)$ can be expressed as in \reff{B(x)=VA(x)V+W(x)}.
This is a consequence of the classical matrix
Positivstellensatz~\cite{HN,KleSch10,SchHol06}.
It can be used to check containment of spectrahedra \cite{KTT15}.

However, in applications, convex sets are often not spectrahedra.
A much more general class of convex sets are projections of spectrahedra,
which we call {\it spectrahedrops}.
The Lasserre type moment relaxations~\cite{LasSDr,Las,NPS}
produce a nested hierarchy of spectrahedrops approximating
and closing down on the (convex hull of a) semialgebraic set.
Many convex semialgebraic sets
are spectrahedrops \cite{HN09,HN,Sce1};
however, not all of them are \cite{Sce2}.

Consider the linear pencils
($\y:=(\y_1,\ldots, \y_r)$,
$\z:=(\z_1,\ldots,\z_s)$)
\be
\label{eq:A B}
\left\{ \baray{rcl}
A(\x, \y) &:=& A_0 + \x_1 A_1 + \cdots + \x_n A_n + \y_1 A_{n+1} + \cdots + \y_r A_{n+r}, \\[1mm]
B(\x,\z)  &:=& B_0 + \x_1 B_1  + \cdots + \x_n B_n + \z_1 B_{n+1} + \cdots + \z_s B_{n+s},
\earay \right.
\ee
where $A_i, B_i$ are all real symmetric matrices.
They define the spectrahedrops
\be\label{eq:PA PB}
\left\{ \baray{rcl}
P_A &:=& \{x\in\R^n:  \exists y\in\R^r ,\, A(x, y)  \succeq 0  \}, \\[1mm]
P_B &:=& \{x\in\R^n: \exists z\in\R^s ,\, B(x,z) \succeq 0 \}.
\earay\right.
\ee
A natural question is: how can we check the containment
\be\label{eq:K_A in K_B}
P_A   \subseteq  P_B ?
\ee
If $P_A \subseteq P_B$, then for all $x\in P_A$
there exists $z\in \re^s$ such that $B(x,z) \succeq 0$.
When there are no lifting variables $\y,\z$,
we have
$P_A=S_A$ and $P_B=S_B$, so
the containment
\eqref{eq:K_A in K_B}
simply reduces to \eqref{eq:S_A in S_B}
and can be certified by \reff{B=VAV+W} or \eqref{B(x)=VA(x)V+W(x)}.
However, when there are lifting variables $y,z$,
\reff{B=VAV+W} and \eqref{B(x)=VA(x)V+W(x)} do not apply,
because the ranges of $y,z$ depend on $x$.
While a Positivstellensatz describing polynomials positive
on spectrahedrops is given in \cite{GN},
to the best of the authors' knowledge,
the question of a satisfactory
certificate for \eqref{eq:K_A in K_B}
is widely open. % in the literature.

\subsection*{Contributions}

In this paper, we study how to check the containment %$K_1 \subseteq K_2$
between projections of two semialgebraic sets
that are given by polynomial matrix inequalities.
By Tarski's transfer principle \cite{BCR},
the projection of a semialgebraic set
is again semialgebraic. However, it is generally
a challenge
to find a concrete description for the projection.
For computational efficiency  we usually need to work directly with
the original semialgebraic descriptions, including the extra variables.
For this purpose, we propose a new matrix Positivstellensatz
that uses {\it lifting polynomials},
which we call a {\it lifted matrix Positivstellensatz}.

Denote by $\mc{S}\re[\x,\y]^{k\times k}$ the space of all real
$k\times k$ symmetric matrix polynomials in
$\x := (\x_1,\ldots, \x_n)$ and $\y :=(\y_1,\ldots,\y_r)$.
The space $\mc{S}\re[\x,\z]^{t\times t}$
is similarly defined, with $\z := (\z_1,\ldots, \z_s)$
and an integer $t>0$.
For $G(\x,\y) \in \mc{S}\re[\x,\y]^{k\times k}$
and $Q(\x,\z) \in \mc{S}\re[\x,\z]^{t\times t}$,
consider the projections of semialgebraic sets defined by them,
\[
\baray{rcl}
P_G &:=& \{x \in \re^n: \, \exists y\in\R^r,\, G(x,y)  \succeq 0  \}, \\
P_Q &:=& \{x \in \re^n: \, \exists z\in\R^s, \,  Q(x,z)  \succeq 0  \}.
\earay
\]
We are interested in a certificate for the containment
\be\label{eq:PG in PQ}
P_G \, \subseteq \, P_Q .
\ee
This task is typically very hard.
For a given $x$, checking the existence of $z$
satisfying $Q(x,z) \succeq 0$ is already very difficult,
as it amounts to verifying whether a polynomial matrix inequality
has a real solution or not.
However, we can easily see that $P_G \subseteq P_Q$ if
there exist polynomials $p_1(\x), \ldots, p_s(\x) \in \re[\x]$ such that
\be \label{P(x):satz:y=p}
\left\{ \baray{c}
Q(\x, \underbrace{ ( p_1(\x), \ldots, p_s(\x) )  }_{\z}  \, ) = \\
 V_0(\x,\y)^TV_0(\x,\y)+{\sum}_{i=1}^\ell  V_i(\x,\y)^T G(\x,\y) V_i(\x,\y) ,
\earay \right.
\ee
for certain matrix polynomials $V_0(\x,\y), \ldots, V_\ell(\x,\y)$.
This is because for every $x$, if there exists $y$ such that
$G(x,y) \succeq 0$ (i.e., $x\in P_G$), then $Q(x,z) \succeq 0$
for $z=(p_1(x), \ldots, p_s(x))$ (i.e., $x\in P_Q$).
The representation \reff{P(x):satz:y=p}
gives a certificate for $P_G \subseteq P_Q$.
When $Q(\x,\z)$ does not depend on $\z$,
\reff{P(x):satz:y=p} is reduced to the classical
matrix Positivstellensatz \cite{KleSch10,SchHol06}.
We call each $p_i$ a {\it lifting polynomial}
and call \reff{P(x):satz:y=p} a
{\it lifted matrix Positivstellensatz certificate}.

When do there exist polynomials $p_1, \ldots, p_s \in \R[\x]$
satisfying \reff{P(x):satz:y=p}?
Is \reff{P(x):satz:y=p} also necessary for $P_G \subseteq P_Q$?
If they do exist, how can one compute $p_i(\x)$ and $V_i(\x,\y)$
satisfying \reff{P(x):satz:y=p}? In this paper, we assume that
the quadratic module generated by $G(\x,\y)$
is archimedean,
which is almost equivalent to the compactness of
the semialgebraic set
$S_G := \{ (x,y)\in\R^n\times \R^r\colon G(x,y)\succeq0\}$ and implies the compactness of
the projection
$P_G$.
The archimedeanness is typically required in a Positivstellensatz.
Our major results are:

\bnum[label={\rm(\Roman*)}]
\item\label{it:I}  When $Q(\x,\z)$ is linear in $\z$,
we show that \reff{P(x):satz:y=p} is also a necessary certificate for
$P_G \subseteq P_Q$, under the following natural condition:
for each $x \in P_G$ there exists $z$ such that $Q(x,z) \succ 0$.
The condition essentially means that
$P_G \subseteq \mbox{int}(P_Q)$, the interior of $P_Q$.
Such a condition is generally required. For instance,
when $Q(\x,\z)$ does not depend on $\z$,
the positivity of $Q(\x)$ on $P_G$ is 
%usually
 required
in the classical matrix Positivstellensatz.
The certificate \reff{P(x):satz:y=p}
can be searched for by solving a semidefinite program,
once the degrees for $p_i, V_j$ are fixed.
This result is given in Theorem~\ref{thm:dropSatz}
in  Subsection~\ref{ssec:3.1}.

\item  When $Q(\x,\z)$ is nonlinear in $\z$,
checking $P_G \subseteq P_Q$ becomes more difficult.
For this case, \reff{P(x):satz:y=p} gives nonlinear equations
for the coefficients of the unknown polynomials $p_i(\x)$,
i.e., \reff{P(x):satz:y=p} is not a convex condition
on the $(p_1, \ldots, p_m)$. Hence,
\reff{P(x):satz:y=p} cannot be checked
by solving a semidefinite program.
This is not surprising, because for a given $x$,
checking the existence of a $z$ satisfying $Q(x,z) \succeq 0$
is already a  difficult problem.

In computation, one often prefers a Positivstellensatz certificate
that can be checked by solving a semidefinite program.
We show that this is possible when
for each fixed $x \in P_G$, the matrix polynomial
$Q(x,\z)$ is sos-concave in $\z$.
Indeed, under the sos-concavity condition,
we prove that \reff{P(x):satz:y=p}
is equivalent to a different Positivstellensatz certificate
using lifting polynomials,
which can again be searched for by solving a semidefinite program.
Under the same condition as in \ref{it:I},
we prove a new lifted matrix Positivstellensatz.
This result is given in Theorem~\ref{thm:Pos:Qsoscvz}
in Subsection~\ref{ssec:3.2}.

A key step in the proofs of the above theorems is the existence of a 
continuous lifting
map $P_G\to S_Q$, where some type of convexity assumption is essential, see
Example \ref{ex:nolift}. Thus
 when $Q(\x,\z)$ is not convex in $\z$, the lifting polynomials might not exist. Hence
  Theorems \ref{thm:dropSatz} and \ref{thm:Pos:Qsoscvz} cannot be extended to the non-convex case.

\item The above lifted matrix Positivstellens\"atze
can be applied to check containment between two spectrahedrops.
Let $P_A, P_B$ be two spectrahedrops as in \eqref{eq:PA PB}.
A certificate for the containment ${P}_A \subseteq P_B$
is the representation
\be \label{B(x,y):Postz:A(x,z)}
\left\{ \baray{c}
B(\x, \underbrace{ ( p_1(\x), \ldots, p_s(\x) )  }_{\z}  \,) =  \\
V_0(\x,\y)^TV_0(\x,\y)+
\sum_{i=1}^\ell  V_i(\x,\y)^T A(\x,\y) V_i(\x,\y)

\earay \right.
\ee
where $p_1, \ldots, p_s$ are scalar polynomials in $\x$
and $V_0, \ldots, V_\ell$ are matrix polynomials in $(\x,\y)$.
%
%Indeed, for each $x\in P_A$ there exists $y$
%such that $A(x,y) \succeq 0$,
%whence $B(x,z) \succeq 0$ for
%$z = (p_1(x), \ldots, p_s(x))$ and hence $x\in P_B$.
%
In Section \ref{sc:conDrop}, we show in Theorem \ref{thm:dropDrop}
that \reff{B(x,y):Postz:A(x,z)}
is also a necessary certificate for $P_A \subseteq P_B$,
under weaker assumptions than in \ref{it:I}.
Indeed, the archimedeanness of the quadratic module of $A(\x,\y)$
can be weakened to the archimedeanness
of its intersection with the ring $\R[\x]^{t\times t}$.
%
%and for each $x \in P_A$ there exists $z$ such that $B(x,z) \succ 0$.
%The latter is almost equivalent to $P_A \subseteq \mbox{int}(P_B)$.
%

\enum

The paper is organized as follows.
Section~\ref{sc:pre} gives  preliminaries
about matrix polynomials and their quadratic modules.
Section~\ref{sc:LMPosz} presents two lifted matrix Positivstellens\"{a}tze,
gives their proofs and several examples.
Section~\ref{sc:conDrop} shows how to apply
the lifted matrix Positivstellensatz to check
containment of spectrahedrops and discusses
the matrix cube problem for spectrahedrops.
Finally, Section~\ref{sc:con} gives conclusions
and discusses some open questions.

\section{Preliminaries}
\label{sc:pre}

This section reviews some preliminary results
about matrix polynomials and
the classical matrix Positivstellensatz.

\subsection{Notation}

\def\cA{\R\cx^{t\times t}}
\def\rrxt{\R\cx^{t\times t}}
\def\srxt{\mbS_t(\R\cx)}
\emph{Matrix polynomials} are elements of the ring
$\R\cx^{k\times k}$ where $\R\cx$ is the ring of polynomials in
$\x :=(\x_1,\ldots,\x_n)$ with coefficients from $\R$.
The space of all $k\times k$ real symmetric matrix polynomials
is denoted as $\mc{S}\R[\x]^{k\times k}$.
Let $I_k$ denote the $k\times k$ identity matrix.
A subset $M \subseteq \mc{S}\R[\x]^{k\times k}$
is called a \emph{quadratic module} if
\[
I_k\in M,\quad M+M\subseteq M\quad\text{and}\quad
a^TMa\subseteq M\text{\, for all } \, a\in \R[\x]^{k\times k}.
\]
Here, the superscript $^T$ denotes the transpose of a matrix.
For a finite set $\Gm\subseteq \mc{S}\R[\x]^{k\times k}$,
define the semialgebraic set
\[
S_\Gm:=\{x\in\R^n \colon \forall g\in \Gm,\, g(x)\succeq 0\}.
\]
The set $\Gm$ generates the following quadratic module in
$\mc{S}\R[\x]^{t\times t}$,
\be \label{df:Qmod:G}
\QM_t(\Gm) \, := \,
\left\{\sum_{i=1}^L p_i^Tg_ip_i
\left|\baray{l}
g_i\in\{I_k\}\cup \Gm, \\
L\in\N,\, p_i\in \R[\x]^{k\times t}
\earay\right.
\right\}.
\ee
In particular, when $\Gm$ is empty,
$\QM_t(\emptyset)$ is the set of all sums of hermitian squares in
$\mc{S}\R[\x]^{t\times t}$, i.e., the \emph{sos matrix polynomials}.
Given a matrix polynomial
$f\in \mc{S}\R[\x]^{t\times t}$ and $S\subseteq\R^n$,
we write $f\succeq 0$ on $S$ if
for all $x\in S$, $f(x)\succeq 0$
(i.e., $f(x)$ is positive semidefinite).
Similarly, by writing $f\succ 0$ on $S$ we mean that
$f(x)\succ 0$, i.e., $f(x)$ is positive definite
for all $x \in S$.
%
%Likewise we use the notation $f\succ 0$ or $f\not\preceq 0$.
%
Clearly, if $f\in \QM_t(\Gm)$ then $f\succeq0$ on $S_\Gm$.
Note that the finite set $\Gm$ can be
replaced by a block-diagonal matrix polynomial.
%the block-diaginal singleton ${\bigoplus}_{g\in\Gm} g$.
Thus there is no harm in assuming that $\Gm=\{G\}$.
In this case we shall write simply $S_G$ and $\QM_t(G)$
for the semialgebraic set and quadratic module generated by $S$, respectively.

In a Positivstellensatz, we usually deal with the case that $S_G$ is compact.
In fact, we often need a slightly stronger assumption
that the quadratic module $\QM_t(G)$ is archimedean.
Here, a quadratic module $M$ of $\mc{S}\R[\x]^{t\times t}$
is said to be \textit{archimedean} if
%\begin{equation}\label{eq:arch}
%\exists R > 0:\; (R-\x_1^2-\cdots -\x_n^2)I_t \in M,
%\end{equation}
%where $I_t$ denotes the $t\times t$ identity matrix.
%
there exists $f\in M$ such that the set $S_f$ is compact.
When $S_G$ is bounded, the archimedeanness
can be enforced by possibly enlarging
$G$ without changing $S_G$.

\subsection{Matrix Positivstellensatz}

For a matrix polynomial $G \in \mc{S}\R[\x]^{k\times k}$,
if $f \in \mc{S}\R[\x]^{t\times t}$ and
$f \succeq 0$ on $S_G$, we might not have $f \in \QM_t(G)$.
To guarantee $f \in \QM_t(G)$,
we typically need that $\QM_t(G)$ is archimedean
(and thus $S_G$  compact) and $f \succ 0$ on $S_G$.
This is the matrix version of Putinar's Positivstellensatz~\cite{Put},
which is given by Scherer \& Hol \cite{SchHol06}.

\begin{theorem}
[\cite{SchHol06}]\label{thm:pos}
Let $G \in \mc{S}\R[\x]^{k\times k}$ be such that
$\QM_t(G)$ is archimedean. For $f \in \mc{S}\R[\x]^{t\times t}$,
if $f\succ 0$ on $S_G$, then $f\in\QM_t(G)$.
\end{theorem}

We refer readers to \cite{HN,KleSch10}
for further refinements of this result, and to
\cite{Cim,HL,Schm} for additional recent results on
Positivstellens\"atze for matrix polynomials.

%
%\subsection{Algorithmic aspects}
%
A matrix polynomial $Q\in \mc{S}\R[\x]^{t\times t}$
is sos if and only if the scalar polynomial
$\y^T Q(\x)\y$ is sos in $(\x,\y)$, where $\y$ is a new
$t$-tuple of variables. This means that sos matrix polynomials
can be checked by solving a semidefinite program.
A more direct procedure (see \cite[Lemma 1]{SchHol06}) is as follows.
When $Q$ has degree $2d$, $Q$ is sos if and only if there exists
a positive semidefinite matrix $Z$ such that
\be\label{eq:SDP}
Q= (u(\x)\otimes I_t)^T Z (u(\x)\otimes I_t),
%\quad Z \succeq 0
\ee
where $\otimes$ is the classical Kronecker product and
$u(\x)$ is the vector of all monomials in $\x$  of degrees $\leq d$.
As \eqref{eq:SDP} is just a set of linear equations
in the entries of a positive semidefinite matrix $Z$,
one can search for a feasible $Z$ by solving a semidefinite program.
More generally, for a given finite set
$\Gm\subseteq \mc{S}\R[\x]^{k\times k}$,
one can check whether or not $Q$ belongs to the
\emph{truncated} quadratic module
\be \label{trun:QM(G)}
\QM_t(\Gm)\big|_{2d} \, := \,
\left\{\sum_{i=1}^L p_i^Tg_ip_i
\left|\baray{l}
g_i\in\{I_k\}\cup \Gm,\,p_i\in \R[\x]^{k\times t}, \\
L\in\N,\,\deg(p_i^Tg_ip_i) \leq 2d
\earay\right.
\right\}.
\ee
This can be done similarly by solving a semidefinite program
\cite[Section 5]{SchHol06}.
For more about the area, we refer to
%\cite{BPR13,deKL,FNT,GPT,GT,HG,HNS,Las01,LasBok,
%Las,Laurent,Laurent2,PS03,PV,RT,Sch09,Scw}
positive polynomials \cite{HG,RT,Sch09},
moment problems \cite{LasBok,Laurent,PV},
convex algebraic geometry \cite{BPR13,FNT,GPT,GT},
polynomial optimization \cite{deKL,HL,Las01,Las,Laurent2,PS03,Scw},
and semidefinite programs \cite{deK,HNS,SDPbk}.

\section{A Lifted Matrix Positivstellensatz}
\label{sc:LMPosz}

In this section, we prove a lifted matrix Positivstellensatz
certifying containment of projections of semialgebraic sets
given by polynomial matrix inequalities.
For $G \in \mc{S}\R[\x,\y]^{k\times k}$
and $Q \in \mc{S}\R[x,y]^{t\times t}$,
consider the projected semialgebraic sets
\begin{align} 
\label{K:G(x)>=0}
P_G  &:= \{x \in \re^n: \exists y\in\R^r,\, G(x,y)  \succeq 0  \}, \\
\label{set:T}
P_Q &:= \{x \in \re^n: \exists z\in\R^s,\,  \, Q(x,z) \succeq 0 \}.
\end{align}
We are going to establish a certificate
for the containment $P_G \subseteq P_Q$.
Our discussion is divided into two cases.
We first analyze the case when
$Q(\x,\z)$ is linear in $\z$, and then
treat the nonlinear case.

\subsection{The case $Q(\x,\z)$ is linear in $\z$}
\label{ssec:3.1}

Suppose $Q(\x,\z)$ is linear in $\z :=(\z_1,\ldots, \z_s)$,
\be \label{Qxz:lin:z}
Q(\x,\z) := Q_0(\x) + \z_1 Q_1(\x) + \cdots + \z_s Q_s(\x),
\ee
where $Q_0(\x), \ldots, Q_s(\x) \in \mc{S}\R[\x]^{t\times t}$
are symmetric matrix polynomials.
A certificate for the inclusion $P_G \subseteq P_Q$
is the existence of polynomials $p_1(\x), \ldots, p_s(\x) \in \rx$
and matrix polynomials $V_i(\x,\y)$ such that
\be \label{Pp(x):in:Q(G)}
\left\{ \baray{c}
Q_0(\x) + p_1(\x) Q_1(\x) + \cdots + p_s(\x) Q_s(\x) = \\[1mm]
 V_0(\x,\y)^TV_0(\x,\y)
 +
 {\sum}_{i=1}^\ell  V_i(\x,\y)^T G(\x,\y) V_i(\x,\y) .
\earay \right.
\ee
Indeed, if $x\in P_G$, then there exists $y\in\R^r$
with $G(x,y)\succeq0$, thus
$Q(x,z) \succeq 0$ for $z=(p_1(x), \ldots, p_s(x))$ by \eqref{Pp(x):in:Q(G)}.
This certifies that $P_G\subseteq P_Q$.

In the following, we show that
\reff{Pp(x):in:Q(G)} is almost necessary for ensuring
$P_G\subseteq P_Q$. Our main conclusion is that
\reff{Pp(x):in:Q(G)} must hold if
$P_G$ is contained in the interior of $P_Q$
(i.e., $P_G \subseteq \mbox{int}(P_Q)$),
under the archimedean condition.
Since $G$ is a matrix polynomial in $(\x,\y)$,
its quadratic module $\QM_t(G)$ is a subset of
$\mc{S}\R[\x,\y]^{t\times t}$.
%
%The archimedeanness of $\QM_t(G)$ requires that
%\[
%\big( R_1 - (x_1^2+\cdots+x_n^2+y_1^2+\cdots+y_r^2) \big) I_t
%\in \QM_t(G)
%\]
%for some scalar $R_1>0$. The intersection
%$\QM_t(G) \cap \mc{S}\R[\x]^{t\times t}$
%is a quadratic module in $\mc{S}\R[\x]^{t\times t}$.
%The archimedeanness of $\QM_t(G) \cap \mc{S}\R[\x]^{t\times t}$ requires that
%\[
%\big( R_2 - (x_1^2+\cdots+x_n^2) \big) I_t
%\in \QM_t(G)
%\]
%for some scalar $R_2>0$.
%
The archimedeanness of $\QM_t(G)$ requires
the existence of $f \in \QM_t(G)$ such that
the set $\{(x,y)\in\R^n\times\R^r: f(x,y) \succeq 0\}$ is compact.
% The intersection $\QM_t(G) \cap \mc{S}\R[\x]^{t\times t}$
% is a quadratic module in $\mc{S}\R[\x]^{t\times t}$.
% Its archimedeanness requires the existence of
% $g \in \QM_t(G) \cap \mc{S}\R[\x]^{t\times t}$ such that
% the set $\{x: g(x) \succeq 0\}$ is compact.

\begin{theorem}\label{thm:dropSatz}
Let $G(\x,\y) \in \mc{S}\R[\x,\y]^{k\times k}$
and let $Q(\x,\z)$ be as in \eqref{Qxz:lin:z}.
Assume that $\QM_t(G)$ is archimedean. If for all $x\in P_G$
there exists $z\in\R^s$ with $Q(x,z)\succ0$,
then there exists a polynomial tuple 
$p(\x)=(p_1(\x), \ldots, p_s(\x))$ such that
$Q(\x,p(\x)) \in \QM_t(G)$, i.e., \eqref{Pp(x):in:Q(G)} holds.
\end{theorem}

\begin{proof}
Since $\QM_t(G)$ is archimedean,
the set $S_G$ is compact, hence so is the projection $P_G$.
% When $\QM_t(G) \cap \mc{S}\R[\x]^{t\times t}$ is archimedean,
% the projection $P_G$ is bounded (while $S_G$ is not necessarily),
% hence $P_G$ is also compact since it is closed.
% So, $P_G$ is compact if either
% $\QM_t(G) \cap \mc{S}\R[\x]^{t\times t}$
% or $\QM_t(G)$ is archimedean.
For each $x \in P_G$, there exists $z$ (depending on $x$,
that is, $z=z(x)$), such that $Q(x, z(x) ) \succ 0$.
Let $\dt = \dt(x) >0$ be such that
$Q(w,z(x) ) \succ 0$ for all $w $ in
the open ball $\mc{B}(x, 2\dt)$
centered at $x$  with radius $2\dt$.
Then, $\{ \mc{B}(x, \dt(x)) \}_{x\in P_G}$ is an open covering for $P_G$.
By  compactness, there exist finitely many of 
these open balls covering $P_G$, say,
\[
P_G \subseteq \bigcup_{i=1}^N  \mc{B}( x^i, \dt(x^i) ).
\]
For each $i$, there exists $\eps_i>0$
such that $Q(w,z(x^i) ) \succeq \eps_i I$ for all
$w \in \mc{B}(x^i, \dt(x^i))$.
Hence, we can choose $\eps>0$ small enough such that
for all $x\in P_G$ there exists $z\in\R^s$ with
$
Q(x,z) \succeq \eps I.
$
Define the function
\be \label{phi(x):P(x,y)>=eps/2}
\baray{rl}
\phi(x):=\arg\min &  z^Tz \\
s.t. &   Q_0(x) + z_1 Q_1(x) + \cdots + z_s Q_s(x) \succeq  \frac{\eps}{2} I.
\earay
\ee
From the above, we can see that the feasible set
of \reff{phi(x):P(x,y)>=eps/2} has nonempty interior for all $x\in P_G$.
Because of the strict convexity of $\z^T\z$,
%the $2$-norm for each $x \in P_G$, 
the minimizer $\phi(x)$ is unique.
Further, the objective is a coercive function, that is,
for every number $\tau >0$, the set
$
\{z :  z^Tz \leq \tau  \}
$
is compact. Hence 
%the objective value function
the optimal value function
$\phi(x)^T\phi(x)$ is continuous in $x$.
This can be inferred from \cite[Theorem~10]{Shapiro97}
or \cite[Theorem~4.1.10]{SDPbk}.

The minimizer function $\phi(x)$ is also a continuous function on
$P_G$, which can be seen as follows.
Suppose $\{ x^k \} \subseteq P_G$ is a sequence such that $x^k \to x \in P_G$.
Then $\| \phi(x^k) \|_2 \to \| \phi(x)\|_2$
by the continuity of the objective function.
Clearly, $\{ \phi(x^k) \}$ is bounded. Let $u$ 
be one of its accumulation points.
Then $\| u \|_2 = \| \phi(x) \|_2$. Clearly,
$u$ is a feasible point corresponding to $x$.
Hence, $u$ is a minimizer for \reff{phi(x):P(x,y)>=eps/2},
and by the uniqueness, $u=\phi(x)$.
So  $\phi(x)$ is a continuous function on $P_G$.
Note that
\[
Q(x, \phi(x) ) \succeq \frac{\eps}{2} I \quad \mbox{ on } \quad P_G.
\]
By the Stone-Weierstra\ss{} theorem (see e.g.~\cite[Theorem 7.32]{Rud}),
$\phi(x)$ can be approximated arbitrarily well
by polynomial functions.
In particular, there exists a polynomial $p(\x)$ such that
\[
Q(x, p(x) ) \succ 0 \quad \mbox{ on } \quad P_G.
\]
That is, $Q(\x,p(\x))$ is symmetric matrix polynomial
that is positive definite on $P_G$.
By the archimedean property of
$\QM_t(G)$,
the classical matrix Positivstellensatz
(see e.g.~\cite{KleSch10,SchHol06}) implies that
\[
Q(\x, p(\x) ) =
V_0(\x,\y)^T V_0(\x,\y)
+
 \sum_i  V_i(\x,\y)^T G(\x,\y) V_i(\x,\y)
\]
for some matrix polynomials $V_i(\x,\y)$.
\end{proof}

\subsection{The case $Q(\x,\z)$ is nonlinear in $\z$}
\label{ssec:3.2}

Denote the set of exponents by
\[
\N^s_{2d} := \{ \af = (\af_1, \ldots, \af_s) \in \Z_{\geq0}^s \mid
\af_1 + \cdots + \af_s \leq 2d \}.
\]
We consider the case that $Q(\x,\z)$ is polynomial in $\z$, say,
\be \label{Qxz:nlz}
Q(\x,\z) := \sum_{   \af \in \N^s_{2d}  }
\z_1^{\af_1} \cdots \z_s^{\af_s}   Q_\af(\x),
\ee
with each $Q_\af(\x) \in \mc{S}\R[\x]^{t\times t}$.
If we parameterize $\z_i$ by a polynomial $p_i(\x)$,
a natural generalization of the certificate \reff{Pp(x):in:Q(G)} is
\be \label{p(x)^af:Posatz}
Q(\x,p(\x))=
\sum_{ \af \in \N^s_{2d} }  \,  p_1(\x)^{\af_1} \cdots p_s(\x)^{\af_s}
Q_\af(\x)  \, \in \, \QM_t(G).
\ee
However, \reff{p(x)^af:Posatz} is nonlinear in the coefficients of
$p = (p_1, \ldots, p_s)$. Generally,
the existence of $p$ satisfying \reff{p(x)^af:Posatz}
cannot be checked by solving a semidefinite program.

Here we propose a convexification of \reff{p(x)^af:Posatz}.
If each product $p_1(\x)^{\af_1} \cdots p_s(\x)^{\af_s}$
is replaced by a new polynomial $p_\af(\x)$,
then \reff{p(x)^af:Posatz} becomes
\be \label{pafP(x):in:G}
\left\{ \baray{c}
\sum_{ \af \in \N^s_{2d} }  p_\af(\x) Q_\af(\x)= \\[1mm]
V_0(\x,\y)^TV_0(\x,\y)
+
 {\sum}_{i=1}^\ell  V_i(\x,\y)^T G(\x,\y) V_i(\x,\y) ,
\earay \right.
\ee
for some matrix polynomials $V_i(\x, \y)$.
However, \eqref{pafP(x):in:G} does not imply
$P_G \subseteq P_Q$ in general.
To remedy this, let
\[
p \, := \, (p_\af)_{ \af \in \N^s_{2d}} ,
\]
and define the matrix polynomial
\[
M( p ) \, := \, ( p_{\af+\bt} )_{ \af, \bt \in \N^s_d }.
\]
In Proposition~\ref{pro:Qsosz:PGinPQ} below,
under some convexity conditions,
we show that \reff{pafP(x):in:G}
is a certificate for $P_G \subseteq P_Q$.
The matrix polynomial $Q(\x,\z)$ is said to be
{\it sos-concave} in $\z$ at a point $x$ if for every $\xi \in \re^t$
the polynomial $\xi^TQ(x,\z)\xi$
is sos-concave in $\z$, i.e., its Hessian
$\nabla^2(\xi^TQ(x,\z)\xi)$ about $\z$ is an sos-matrix polynomial in $\z$.
We refer to \cite{njwPMI} for more on 
sos-concavity/convexity
of matrix polynomials.

\begin{prop} \label{pro:Qsosz:PGinPQ}
Let $G(\x,\y) \in \mc{S}\R[x,y]^{k\times k}$
and let $Q(\x,\z)$ be as in \eqref{Qxz:nlz}.
Assume $Q(x,\z)$ is sos-concave in $\z$ at  every $x\in P_G$.
If a polynomial tuple $p$ satisfies \eqref{pafP(x):in:G}
and $M(p) \succeq 0$ on $P_G$, then $P_G \subseteq P_Q$.
\end{prop}

\begin{proof}
Define a matrix polynomial in $\x=(\x_1,\ldots, \x_n)$
and $\w = (\w_\af)_{\af \in \N^s_{2d} }$ as
\[
F(\x,\w):= \sum_{ \af \in \N^s_{2d} }   \w_\af Q_\af(\x).
\]
Pick an arbitrary $x\in P_G$. Let $w_\af = p_\af(x)$
(note $w_0 = 1$), then
\[
F(x,w) \succeq 0, \quad M(w) \succeq 0.
\]
For an arbitrary $\xi \in \re^t$,
the polynomial $q(\z):=\xi^TQ(x,\z)\xi$ is sos-concave in $\z$.
Let $u = (w_1, \ldots, w_s)$, then one can show that
(see e.g.~\cite[Theorem~9]{HN} or \cite[Theorem~2.6]{Las09})
\[
q(u) \geq
\sum_{ \af \in \N^s_{2d} }   \w_\af \xi^T Q_\af(x) \xi
=\xi^T F(x,w) \xi \geq 0.
\]
Since $q(u) = \xi^T Q(x,u)\xi \geq 0$ and $\xi$ is arbitrary,
we can conclude that $Q(x,u) \succeq 0$, i.e., $x \in P_Q$.
The above can also be deduced from the results in \cite{njwPMI}.
%
%Let $z = (w_1, \ldots, w_n)$.
%Since $Q(\x,\z)$ is sos-concave in $\z$
%for every $x \in P_G$, $Q(x,z) \succeq 0$ if and only if
%\cite[Theorem~2.2]{njwPMI}
%\[
%F(x,z) \succeq 0, \, M(w) \succeq 0, \, w_i = z_i.
%\]
%So, $Q(x,z) \succeq 0$.
%
Since $x\in P_G$ was arbitrary,
we conclude that $P_G \subseteq P_Q$.
\end{proof}

In the following, we show that \reff{pafP(x):in:G}
is almost a necessary certificate for $P_G \subseteq P_Q$
under  conditions similar to those in
Theorem~\ref{thm:dropSatz} and
Proposition~\ref{pro:Qsosz:PGinPQ},
and under an additional sos-concavity condition.

\begin{theorem}  \label{thm:Pos:Qsoscvz}
Let $G(\x,\y) \in \mc{S}\R[x,y]^{k\times k}$
and let $Q(\x,\z)$ be as in \eqref{Qxz:nlz}.
Assume that  $\QM_t(G)$ is archimedean.
If for every $x \in P_G$, $Q(x,\z)$ is sos-concave in $\z$,
and there exists $z$ such that $Q(x,z) \succ 0$,
then there exist polynomials
$ p_\af \in \R[\x]$ ($\af \in \N^s_{2d}$)
such that \eqref{pafP(x):in:G} holds
and $M(p)$ is an sos matrix polynomial.
\end{theorem}

\begin{proof}
The proof is similar to the one for Theorem~\ref{thm:dropSatz}.
First, we can similarly prove that there exists $\eps>0$ such that
for all $x \in P_G$ there exists $z$ with $Q(x,z) \succeq \eps I$.
Consider the optimization problem
\be \label{min:yTy:sos-y}
\min \quad z^Tz  \quad s.t. \quad Q(x, z) \succeq \frac{\eps}{2} I.
\ee
For each $x\in P_G$, the feasible set of \reff{min:yTy:sos-y}
has nonempty interior. It has a unique minimizer,
which we also denote by $\phi(x)$. Note that
\reff{min:yTy:sos-y} is a convex optimization problem
and the objective is coercive. Furthermore,
$\phi(x)$ is a continuous function on $P_G$.
By the Stone-Weierstra\ss{} theorem,
there exists a polynomial tuple
$q(\x) := (q_1(\x), \ldots, q_s(\x))$
such that $Q(x, q(x)) \succeq \frac{\eps}{4}I$ on $P_G$.
By the archimedean property and the classical matrix
Positivstellensatz (see e.g.~\cite{KleSch10,SchHol06}), we get
\[
Q(\x, q(\x)) \, \in \, \QM_t(G).
\]
For each $\af$, let $p_\af = q^\af$, then
$
M(p) \, = \, [q]_d [q]_d^T.
$
In the above, $[q]_d$ is the vector of all monomials in $q$
of degrees $\le d$. Clearly, $M(p)$
is an sos matrix polynomial and the proof is complete.
\end{proof}

\begin{exm}\label{ex:nolift}
We want to point out that a lifting continuous map $\phi:P_G\to \mbox{int}(S_Q)$ 
need not exist without some convexity assumptions on $Q$. 
Hence
  Theorems \ref{thm:dropSatz} and \ref{thm:Pos:Qsoscvz} do not generalize to the non-convex case.
Here are two simple examples.
\begin{enumerate}[label={\rm(\alph*)}]
\item
Form $S_Q$ by rotating the semialgebraic set defined as the part of the hyperbola $x^2-z^2\geq1$
lying inside $x^2\leq4$ by $60^\circ$ about the origin.
That is, 
\[
Q(\x,\z):=\diag
\left(
-4-(-\sqrt3\x+\z)^2+(\x+\sqrt3\z)^2,\,
16-(\x+\sqrt3\z)^2
\right).
\]

\def\nos{100}
\begin{center}
\ifx\foo\undefined
\begin{tikzpicture}[domain=-3.3:3.3,scale=1.25] \else
\begin{tikzpicture}[domain=-3.3:3.3,scale=1.5] 
\fi
\draw[very thin,color=gray!25] (-2.9,-2.9) grid (2.9,2.9);
\fill[color=ogreen, domain=1.31696:-1.31696, samples=\nos,rotate=60,fill opacity=0.2] plot ({cosh(\x)},{sinh(\x)});
\fill[color=ogreen, domain=1.31696:-1.31696, samples=\nos,rotate=60,fill opacity=0.2] plot ({-cosh(\x)},{sinh(\x)});
\draw[color=black, domain=-1.5:1.5, thick,  samples=\nos,rotate=60] plot ({cosh(\x)},{sinh(\x)});
\draw[color=black, domain=-1.5:1.5, thick,  samples=\nos,rotate=60] plot ({-cosh(\x)},{sinh(\x)});

\draw[color=black, domain=-1.31696:1.31696, ultra thick,  samples=\nos,rotate=60] plot ({cosh(\x)},{sinh(\x)});
\draw[color=black, domain=-1.31696:1.31696, ultra thick,  samples=\nos,rotate=60] plot ({-cosh(\x)},{sinh(\x)});

\draw[color=black, ultra thick,  samples=\nos,rotate=60] (2,1.73205) -- (2,-1.73205);
\draw[color=black, ultra thick,  samples=\nos,rotate=60] (-2,1.73205) -- (-2,-1.73205);

\draw[->] (-3.1,0) -- (3.1,0) node[right] {$\x$}; \draw[->] (0,-3.1) -- (0,3.1) node[above] {$\z$};
\filldraw[color=black] (1,0) circle (0.02) node[below] {$1$};
\filldraw[color=black] (2,0) circle (0.02) node[below] {$2$};
\filldraw[color=black] (0,1) circle (0.02) node[above left] {$1$};
\filldraw[color=black] (0,2) circle (0.02) node[above left] {$2$};
\filldraw[color=black] (2.5,0) circle (0.04) node[below] {$\frac52$};
\filldraw[color=black] (-2.5,0) circle (0.04) node[below] {$-\frac52$};
\filldraw[color=black] (.5,0) circle (0.04) node[below] {$\frac12$};

\draw[dashed, color=gray!50] (0.5, -2.59808) -- (.5,0);
\draw[dashed, color=gray!50] (2.5, 0.866025) -- (2.5,0);
\draw[dashed, color=gray!50] (-2.5, -0.866025) -- (-2.5,0);

\draw[ultra thick,color=dred] (-2.5,0) -- (2.5,0);

\draw [->, thick ] (-2.5,1.5) node [above] {\textcolor{dred}{$P_Q$}} to [bend  left=35] (-1.5,0.1) ;

\end{tikzpicture} \vspace{-2mm}
\end{center}~\centerline{\textcolor{ogreen}{$S_Q$}}

\bigskip

Then $P_Q=\left[-\frac52,\frac52\right]$.
The maximal $x$-coordinate of a point in the bottom
component of $S_Q$ is $\frac{1}{2}$, so by letting
$P_G=[-1,1]$ it is clear that each point $x$ in $P_G$ can be lifted to a 
point $(x,z)\in S_Q$ with $Q(x,z)\succ0$, but there is no lifting continuous map
$P_G\to S_Q$.
\item
The same phenomena can occur even with an
S-shape connected  $S_Q\subseteq\R^2$. 
Let $S_Q$ be the band around a cubic curve,

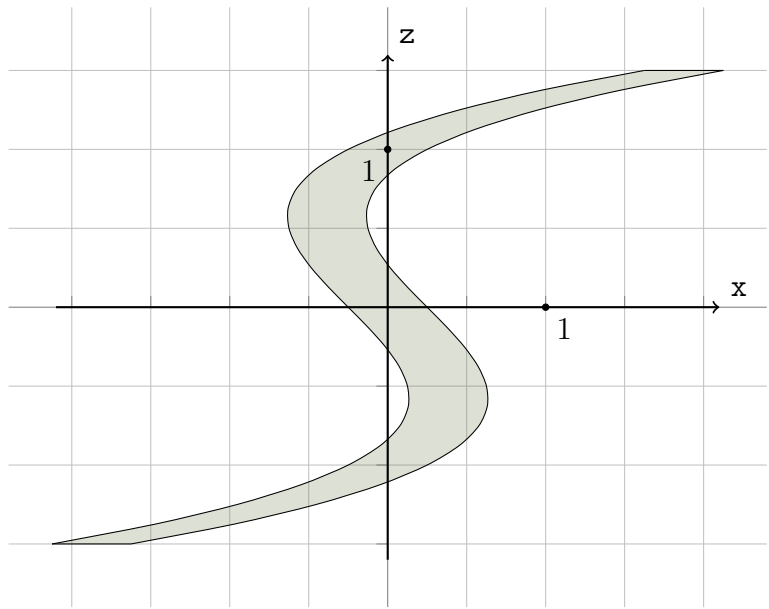
\ifx\foo\undefined
\hspace{.05cm}
\else
\hspace{2cm}
\fi
%\begin{center}
\begin{tikzpicture}[domain=-1.5:1.5] 
\begin{axis}[thick,smooth,no markers,
%domain=-1.5:1.5,
rotate=90, 
scale=1.4,
x=1.5cm,y=1.5cm,
xmin=-1.9,xmax=1.9,ymin=-2.4,ymax=2.4,
    grid=both,
    grid style={line width=.1pt, draw=gray!10},
    major grid style={line width=.3pt,draw=gray!50},
    axis lines*=middle,
    minor tick num=0,
    axis line style={draw=gray},
    xticklabels={,,},
    yticklabels={,,}
]

        \addplot+[name path=A,black] {.25-1*((\x)*((\x)^2-1))};
        \addplot+[name path=B,black] {-.25-1*((\x)*((\x)^2-1))};

        \addplot[color=ogreen,opacity=.2] fill between[of=A and B];

 \addplot [black]  coordinates {(1.5,-1.625)  (1.5,-2.125)};

 \addplot [black]  coordinates {(-1.5,1.625)  (-1.5,2.125)};
        % \addplot[very thin,color=gray!25] draw (-1.9,-1.9) grid (1.9,1.9);
         \filldraw (axis cs:1,0) circle[radius=2] node[below left] {$1$};
         \filldraw (axis cs:0,-1) circle[radius=2] node [below right] {$1$};

\draw[thick,->] (axis cs:-1.6,0) -- (axis cs:1.6,0) node[above right] {$\z$}; 
\draw[thick,->] (axis cs:0,2.1) -- (axis cs:0,-2.1) node[above right] {$\x$}; 

\end{axis}

\end{tikzpicture} %\vspace{-3mm}
%\end{center}~
\[
S_Q=\Big\{(x,z)\in\R^2 \colon \big|x-z(z^2-1)\big|\leq\frac14,\, |z|\leq\frac32
\Big\}.
\]
%\centerline{\textcolor{ogreen}{$S_Q$}}

We have $P_Q=\left[-\frac{17}8,\frac{17}8\right]$.
As before, 
each point $x\in
P_G:=[-1,1]$ 
admits a lift to a 
point $(x,z)\in \mbox{int}(S_Q)$, but there is no lifting continuous map 
$P_G\to S_Q$.
\end{enumerate}
\end{exm}

\subsection{Some examples}

In the following, we give some examples of the lifted matrix
Positivstellensatz proved in Theorems~\ref{thm:dropSatz}
and \ref{thm:Pos:Qsoscvz}. The notation $e_i$
denotes the standard $i$th unit vector, i.e.,
its $i$th entry is one and all other entries are zero.

\begin{exm}
Consider the matrix polynomials
\[
G(\x,\y) = \bbm 1-\y-\x_1^2 &  \x_1\x_2 \\ \x_1\x_2 & \y -\x_2^2 \ebm, \quad
Q(\x,\z) = \bbm 1+\z\,\x_2 &   \z-2\,\x_1 \\  \z-2\,\x_1   & 1-\z\,\x_2 \ebm.
\]
Then $ P_G=\{(x_1,x_2)\in\R^2\colon |x_1\pm x_2|\leq1\}$,
and is contained in
\[
P_Q =\{(x_1,x_2)\in\R^2\colon 1+x_2^2-4x_1^2x_2^2\geq0\}.
\]
The quadratic module $\QM_2(G)$ is archimedean since
\begin{multline} \nn
3-\x_1^2-\x_2^2-2y^2 = e_1^TG(\x,\y)e_1 + e_2^TG(\x,\y)e_2 +
 (1-\y)^2+\\ \x_2^2(1-\y)^2+ \x_1^2(1+\y)^2+
e_1^TG(\x,\y)e_1 (1+y)^2 + e_2^TG(\x,\y)e_2(1-\y)^2.
\end{multline}
The polynomial $p_1$ in Theorem~\ref{thm:dropSatz}
can be chosen to be $\x_1$; then
\[
Q(\x,p(\x)) = \half \bbm \x_1+\x_2 & -1 \\ -1 & \x_1-\x_2 \ebm^2+
\frac{1-\x_1^2-\x_2^2}{2}I.
\]
A certificate of the form
\reff{Pp(x):in:Q(G)} for $P_G \subseteq P_Q$ is
\[
\ell = 4, \quad
V_0  = \frac{1}{\sqrt{2}} \bbm \x_1+\x_2 & -1 \\ -1 & \x_1-\x_2 \ebm,
\]
\[
V_1  = \bbm \frac{1}{\sqrt{2}} &  0 \\ 0  &  0 \ebm, \quad
V_2  = \bbm 0  &  0 \\ \frac{1}{\sqrt{2}} &  0 \ebm,
V_3  = \bbm 0  & \frac{1}{\sqrt{2}}  \\ 0  &  0 \ebm, \quad
V_4  = \bbm 0  &  0 \\ 0  & \frac{1}{\sqrt{2}}  \ebm.
\]
%The matrix polynomials $V_i(\x,\y)$ do not depend on $\y$.
\end{exm}

\begin{exm}\label{ex:3.5}
We present an example where the assumptions of
Theorem \ref{thm:dropSatz} are not met, but the
conclusion still holds.
Consider the matrix polynomials
\[
G(\x,\y)  = \bbm  1-\x_1^2 &  \x_1+\x_2 &  \x_2^2 \\
\x_1+\x_2  & 0 & \x_1+\x_2 \\ \x_2^2 & \x_1+\x_2 &  \y \ebm, \quad
Q(\x,\z)  = \bbm 1+2\eps+\x_2 &  \x_1^2  \\  \x_1^2  & \z \ebm,
\]
for $\eps > 0$. The projection set
$
P_G=\{ (x_1,-x_1)\in\R^2\colon -1<x_1<1\}.
$
It is bounded but not closed.
The intersection $\QM_2(G) \cap \R[\x]$ is archimedean, because
\[
(2-\x_1^2-\x_2^2)  = e_1^T G(\x,\y)e_1 +
\bbm 1 \\ \half(\x_1-\x_2) \\0\ebm^T G(\x,\y)
\bbm 1 \\ \half(\x_1-\x_2) \\0 \ebm.
\]
However, the quadratic module $\QM_2(G)$
itself is not archimedean, since $S_G$ is unbounded.
The lifting polynomial $p_1$ %in Theorem~\ref{thm:dropSatz}
can be chosen as $\eps^{-1} \x_1^2$, then
\[
Q(\x,p(\x)) = \x_1^2 \bbm \eps &  1 \\ 1  & \eps^{-1} \ebm +
(1+\x_2+\eps+ \eps(1-\x_1^2) ) \bbm 1 & 0 \\ 0  &  0 \ebm.
\]
Note the following representations:
\[
1+\x_2+\eps = \bbm \sqrt{\eps} \\ \sqrt{4\eps}^{-1} \\ 0 \ebm^T
G(\x,\y) \bbm \sqrt{\eps} \\ \sqrt{4\eps}^{-1} \\ 0 \ebm +
1-\x_1 + \eps \x_1^2,
\]
\[
1-\x_1^2 = e_1^T G(\x,\y) e_1, \quad
1-\x_1 = \frac{1-\x_1^2}{2} + \frac{(\x_1-1)^2}{2}.
\]
A certificate of the form
\reff{Pp(x):in:Q(G)} for $P_G \subseteq P_Q$
is that $\ell=3$ and
\[
V_0  = \bbm \sqrt{\eps} \x_1 &  \sqrt{\eps}^{-1} \x_1 \\
\sqrt{\eps} \x_1 & 0  \\
\frac{\x_1 -1}{ \sqrt{2} } & 0 \ebm, \,
V_1 = \bbm \sqrt{\eps}  & 0 \\ 0 & 0 \\ 0 &  0 \ebm, \,
V_2  =
\bbm \sqrt{\eps} & 0 \\ \sqrt{4\eps}^{-1} & 0 \\ 0 & 0 \ebm, \,
V_3  = \bbm \frac{1}{ \sqrt{2} }  & 0 \\ 0 & 0 \\ 0 &  0 \ebm.
\]
%The matrix polynomials $V_i(\x,\y)$ do not depend on $\y$.
\end{exm}

In Theorem~\ref{thm:dropSatz},  if $\QM_t(G)$ is not archimedean,
its conclusion might not hold. The following is such an example.

\begin{exm}  \label{exm:3.6}
Consider the matrix polynomial
\[
%G(\x,\y)=\diag\big(\y^2(1-\x^2)-1,\ 2-\x^2 \big).
G(\x,\y)=\bbm \y^2(1-\x^2)-1 & 0 \\ 0 & \ 2-\x^2 \ebm.
\]
Clearly, $P_G=(-1,1)$ is bounded. The intersection
$\QM_1(G)\cap \R\cx$ is archimedean, since
$2-x^2\in \QM_1(G)\cap \R\cx$.
However, the quadratic module $\QM_1(G)$
itself is not archimedean, because $S_G$ is unbounded.
We claim that $\QM_1(G)\cap \R\cx$
is generated by the polynomial $2-\x^2$.
For every $g(\x) \in \QM_1(G)\cap \R\cx$, we can write
\be\label{eq:sigmay}
g(\x) =\sigma_0+\sigma_1 \cdot (\y^2(1-\x^2)-1)+\sigma_2 \cdot (2-\x^2)
\ee
for sos polynomials $\sigma_j\in\R[\x,\y]$.
%
%Write $\sigma_j=\sum_{i=0}^{d} a_i^{(j)}(\x) \y^j$,
%where $d$ is the highest degree in $\y$.
%{\bf I do not follow the below proof. More details are required!!!}
%Then, $a_d^{(j)}(\x)$ is sos for $j=0,2$, $a_{d-2}^{(2)}(\x)$ is sos and $a_{d}^{(2)}(\x)=a_{d-1}^{(2)}(\x)=0$.
%Hence
%\[
%a_d^{(0)}(\x)+a_{d-2}^{(1)}(\x) (1-\x^2)+a_d^{(2)}(\x)(2-\x^2)=0 \]
%which implies that $\sigma_j\in\R\cx$ for all $j$ and $\sigma_1=0$
%as it is the only one interacting with $\y$.
%
Note that $g(\x)$ does not depend on $\y$.
To cancel $\y$ on the right hand side of
\eqref{eq:sigmay},
we must have $\sig_1 = 0$. Similarly, $\sig_0$ and $\sig_2$
 cannot depend on $\y$.
We can conclude that $g\in\QM_1(2-\x^2)\subseteq\R\cx$.
Finally, for each $\lambda\in(1,2)$, the polynomial
$\lambda-\x^2$ is positive on $P_G$, but it does not belong to
$\QM_1(G)\cap \R\cx$. The conclusion of Theorem \ref{thm:dropSatz}
fails for this example, because $\QM_1(G)$ is not archimedean.
\end{exm}

\begin{exm}
Consider the matrix polynomials
\[
G(\x,\y) =  \bbm  \x_1 & \y  & \x_1 \\  \y & \x_2 &  \x_2 \\ \x_1 & \x_2 &  1 \ebm,
\quad
Q(\x,\z) = \bbm \x_1+2\,\z_1-\z_1^2 &  \z_1\z_2 &  \x_2 \\
\z_1\z_2 & \x_2+2\,\z_2-\z_2^2 & \x_1 \\
 \x_2 & \x_1  &  1 \ebm.
\]
Note that $P_G=[0,1]^2$ and $\QM_3(G)$
is archimedean, because
\[
1-\x_1^2= \bbm 1 \\ 0 \\ -1 \ebm G \bbm 1 \\ 0 \\ -1 \ebm +
\bbm 1 \\ 0 \\ -\x_1 \ebm G \bbm 1 \\ 0 \\ -\x_1 \ebm,
\]
\[
1-\x_2^2= \bbm 0 \\ 1 \\ -1 \ebm G \bbm 0 \\ 1 \\ -1 \ebm +
\bbm 0  \\ 1  \\ -\x_2 \ebm G \bbm 0 \\ 1 \\ -\x_2 \ebm.
\]
As in Example \ref{ex:3.5} this also yields $1-\x_i\in \QM_3(G)$. Hence
\[
2-\y^2 = (1-\x_2) + (1-\x_1)\y^2 +
\frac12 \bbm 1\\ -\y \\ 1\ebm^T G \bbm 1\\ -\y \\ 1\ebm
+\frac12 \bbm 1\\ -\y \\ -1\ebm^T G \bbm 1\\ -\y \\ -1\ebm\in\QM_3(G).
\]
The matrix polynomial  $Q(\x,\z)$ is sos-concave in $\z$.
The polynomials $p_i$ in Theorem~\ref{thm:Pos:Qsoscvz}
can be chosen as
\[
p_\af = \x_2^{\af_1} \x_1^{\af_2}, \quad
 \af =(\af_1, \af_2) \in \N^2_2 .
%\z_1 = \x_2, \, \z_2 = \x_1.
\]
Clearly,
\[
M(p)  = \bbm 1 & \x_2 &  \x_1 \\ \x_2  & \x_1^2 & \x_2\x_1
\\ \x_1 & \x_1\x_2 & \x_1^2 \ebm
= \bbm 1 \\ \x_2 \\  \x_1 \ebm \bbm 1 \\ \x_2 \\ \x_1 \ebm^T
\]
is sos. We have
\[
Q(\x,p(\x)) = \bbm \x_2 \\ \x_1 \\ 1 \ebm\bbm \x_2 \\ \x_1 \\ 1 \ebm^T +
\bbm \x_1+2(\x_2-\x_2^2) & 0 & 0 \\ 0 & \x_2 + 2(\x_1-\x_1^2) & 0 \\ 0 & 0 & 0 \ebm.
\]
Observe that
\[
\x_1 = e_1^TG(\x,\y)e_1, \quad  \x_2 = e_2^TG(\x,\y)e_2,
\]
\[
\x_1-\x_1^2 = \bbm 1 \\ 0 \\ -\x_1 \ebm^T G(\x,\y) \bbm 1 \\ 0 \\ -\x_1 \ebm, \quad
\x_2-\x_2^2 = \bbm 0 \\ 2\\ -\x_2 \ebm^T G(\x,\y) \bbm 0 \\ 2\\ -\x_2 \ebm.
\]
A certificate of the form
\reff{pafP(x):in:G} for $P_G \subseteq P_Q$ is that
\[
\ell = 4, \quad
V_0(\x,\y) = \bbm \x_2 & \x_1 & 1 \ebm,
\]
\[
V_1  = \bbm 1 & 0 & 0 \\ 0 & 0  & 0 \\ 0 & 0 & 0 \ebm, \,
V_2  = \bbm 1 & 0 & 0 \\ 0 & 0  & 0 \\ -\x_1 & 0  & 0\ebm, \,
V_3  = \bbm 0 & 0 & 0 \\ 0 & 1  & 0 \\ 0 & 0 & 0 \ebm,\,
V_4  = \bbm  0 & 0 &  0\\  0 & 2 &  0\\ 0 & -\x_2  & 0 \ebm.
\]
%The matrix polynomials $V_i(\x,\y)$ do not depend on $\y$.
\end{exm}

\section{Containment of Spectrahedrops}
\label{sc:conDrop}

In this section, we show how to apply the lifted matrix Positivstellensatz
developed in Section~\ref{sc:LMPosz} to check the containment of
spectrahedrops. Recall that a spectrahedrop is
the projection of a spectrahedron.
Convex semialgebraic sets are often spectrahedrops \cite{Las,HN,Sce1},
although not all of them are \cite{Sce2}.

Consider two spectrahedrops
\[
P_A :=\{x:  \exists y ,\, A(x, y)  \succeq 0  \}, \quad
P_B := \{x: \exists z ,\, B(x,z) \succeq 0 \},
\]
where $A(\x,\y) \in \mc{S}\R[\x,\y]^{k\times k}$,
$B(\x,\z) \in \mc{S}\R[\x,\z]^{t\times t}$ are linear pencils
as in \eqref{eq:A B}.
% \[
% A(x, z) := A_0 + x_1 A_1 + \cdots + x_n A_n + z_1 A_{n+1} + \cdots + z_r A_{n+r},
% \]
% \[
% B(x,y) := B_0 + x_1 B_1  + \cdots + x_n B_n + y_1 B_{n+1} + \cdots + y_s B_{n+s}.
% \]
% All $A_i$ and $B_i$ are given symmetric matrices.
An important question of wide applications is how to
check the containment $P_A \subseteq  P_B$?
When $P_A, P_B$ are spectrahedra (i.e.,
there are no lifting variables $\y,\z$),
there exist Positivstellens\"atze 
certifying the containment~\cite{HKM13,KTT13,KTT15}.
In this section, we present a certificate for the containment
when there are lifting variables $\y,\z$.
Theorem~\ref{thm:dropSatz} can be applied.
In fact, when the included set is a spectrahedrop,
the assumptions in Theorem~\ref{thm:dropSatz} can be weakened.
Recall that the intersection
$\QM_t(A)\cap\mc{S}\R[\x]^{t\times t}$
is archimedean if there exists $f(\x) \in \QM_t(A)\cap\mc{S}\R[\x]^{t\times t}$
such that $f(x) \succeq 0$ defines a compact set in $\re^n$.
The archimedeanness of $\QM_t(A)\cap\mc{S}\R[\x]^{t\times t}$
implies the boundedness, but not the closedness, of $P_A$.
Clearly, the archimedeanness of $\QM_t(A)$ implies that
$\QM_t(A)\cap\mc{S}\R[\x]^{t\times t}$ is archimedean
and $P_A$ is closed, but not vice versa;
cf. Example~\ref{exm:3.6}.

\begin{theorem}\label{thm:dropDrop}
Let $A(\x,\y)$ and $B(\x,\z)$ be linear pencils as in \eqref{eq:A B}.
Assume that $\QM_t(A)\cap\mc{S}\R[\x]^{t\times t}$
is archimedean. If there is $\eps>0$ such that
for each $x \in P_A$ there exists $z$ with $B(x,z) \succeq\eps I$,
then there exists a tuple $f(\x):=(f_1(\x), \ldots, f_s(\x))$
of polynomials in $\rx$ such that
\be \label{Bx:fxz:QA2}
B(\x,f(\x)) =
B_0 + \sum_{i=1}^n \x_i B_i  +  \sum_{j=1}^s f_j(\x) B_{n+j}
\, \in \, \QM_t( A(\x,\y)  ).
\ee
\end{theorem}

\begin{proof}
For brevity, let us write $M:=\QM_t(A)\cap\mc{S}\R[\x]^{t\times t}$.
We claim that the positivity set of $M$,
\[
S_M:=\{x\in\R^n\colon \forall g\in M,\, g(x)\succeq0\}
\]
equals the closure $\overline{P_A}$.
The inclusion $\overline{P_A}\subseteq S_M$ is clear. For the converse, assume
$u \in S_M\setminus\overline{P_A}.$ Since $\overline{P_A}$ is convex,
there is a linear polynomial $\ell(\x)$
satisfying $\ell(\x)\geq\af>0$ on $\overline{P_A}$
for some $\af$, and $\ell(u)<0$.  In particular, $\ell(\x)\geq\af>0$
on $S_A$. So, by the linear Positivstellensatz \cite[Corollary 4.2.4]{KSmor},
$\ell(\x)\in \QM_1(A)\cap\R\cx$. This implies that
$\ell(\x) I \in M$,  leading to
the contradiction $\ell(u)\geq0$.

The rest of the proof is the same as for Theorem \ref{thm:dropSatz}.
We can continuously choose for each $x\in P_A$ a point
$z=z(x)\in\R^s$ %with $(x,z)\in S_B$
satisfying
$B(x,z)\succeq\frac{\eps}2 I$. By the Stone-Weierstra\ss{} theorem,
there is a tuple of polynomials
$f(\x) := ( f_1(\x),\ldots,f_s(\x) )$ such that
$B(\x,f(\x) )\succ0$ on $\overline{P_A}=S_M$.
Since $M$ is archimedean,
the matrix Positivstellensatz (see e.g.~\cite{KleSch10})
implies $B(\x,f (\x))\in\QM_t(A)$, as desired.
\end{proof}

In Theorem~\ref{thm:dropDrop}, we assume the existence of
a uniform $\eps>0$ such that
for all $x \in P_A$ there exists $z$ with $B(x,z) \succeq\eps I$.
This is inconvenient to check in applications.
However, the condition can be weakened to
%the same one as in Theorem \ref{thm:dropSatz}
$B(x,z)\succ0$
when $P_A$ is closed.

\begin{corollary}\label{cor:dropDrop1}
Let $A(\x,\y)$ and $B(\x,\z)$ be linear pencils as in \eqref{eq:A B}.
Assume that
$\QM_t(A)\cap\mc{S}\R[\x]^{t\times t}$
is archimedean and $P_A$ is closed.
If for each $x \in P_A$ there exists $z$ with $B(x,z) \succ0$,
then there exist a tuple $f(\x):=(f_1(\x), \ldots, f_s(\x))$
of polynomials in $\rx $ such that
\eqref{Bx:fxz:QA2} holds.
% \be \label{Bx:fxz:QA2'}
% B(\x,f(\x)) =
% B_0 + \sum_{i=1}^n \x_i B_i  +  \sum_{j=1}^s f_j(\x) B_{n+j}
% \, \in \, \QM_t( A  ).
% \ee
\end{corollary}

\begin{proof}
%If $S_A$ is bounded, then $P_A$ is compact
%and $\QM_t(A)$ is archimedean \cite{KSmor}.
If $\QM_t(A)\cap\mc{S}\R[\x,\z]^{t\times t}$ is archimedean,
then $P_A$ is bounded. Hence, $P_A$ is compact since it is also closed.
An $\eps>0$ satisfying Theorem~\ref{thm:dropDrop}
can be found similarly as in the proof of
Theorem~\ref{thm:dropSatz}. Therefore,
the  corollary follows from Theorem \ref{thm:dropDrop}.
\end{proof}

Clearly, \reff{Bx:fxz:QA2} implies that $P_A \subseteq P_B$.
Theorem~\ref{thm:dropDrop} essentially says that
\reff{Bx:fxz:QA2} is a necessary certificate
when $P_A$ is contained in the interior of $P_B$,
i.e., $P_A \subseteq \mbox{int}(P_B)$.
Note that in \reff{Bx:fxz:QA2} the polynomials
$f_i$ only depend on $\x$.

\begin{exm}
Consider the linear pencils
\[
\begin{split}
A(\x,\y) &:= \diag\Big(
\bbm  \y_1  & \x_1 \\ \x_1 & 1 \ebm,
\bbm  \y_2 & \x_2 \\ \x_2 & 1 \ebm,
\bbm  1+\y_1 & \y_2 \\ \y_2 & 1-\y_1 \ebm
\Big),
\\
B(\x,\y)& := \bbm  1 & \x_1 & \z \\ \x_1 &  1  & \x_2  \\ \z  &  \x_2  & 1 \ebm.
\end{split}
\]
The spectrahedrop $P_A$ is the unit $4$-norm ball $\{x_1^4+x_2^4 \leq 1\}$,
while $P_B$ is the unit square $[-1,1]^2$.
Clearly, $P_A \subseteq P_B$.
We give a certificate of the form \reff{Bx:fxz:QA2} for this inclusion.
The polynomial $f_1$ in Theorem~\ref{thm:dropDrop}
can be chosen as $\x_1\x_2$. Note that
\[
B(\x,f(\x)) =
\bbm  \x_1 \\ 1 \\  \x_2  \ebm   \bbm  \x_1 \\ 1 \\  \x_2  \ebm^T +
\bbm  1-\x_1^2 &   &   \\  &  0  &    \\    &    & 1-\x_2^2 \ebm,
\]
\[
\begin{split}
1-\x_1^2 &=
\bbm 1 \\ -\x_1 \ebm^T\bbm  \y_1  & \x_1 \\ \x_1 & 1 \ebm \bbm 1 \\ -\x_1 \ebm +
\bbm 0 \\ 1 \ebm^T\bbm  1+\y_1 & \y_2 \\ \y_2 & 1-\y_1 \ebm \bbm 0 \\ 1 \ebm,\\
1-\x_2^2 &=
\bbm 1 \\ -\x_2 \ebm^T\bbm  \y_2 & \x_2 \\ \x_2 & 1 \ebm \bbm 1 \\ -\x_2 \ebm +
\half \bbm 1 \\ -1 \ebm^T\bbm  1+\y_1 & \y_2 \\ \y_2 & 1-\y_1 \ebm \bbm 1 \\ -1 \ebm.
\end{split}
\]
The certificate for the inclusion $P_A \subseteq P_B$
of the form~\reff{B(x,y):Postz:A(x,z)},
or equivalently \reff{Bx:fxz:QA2}, is
\[
\baray{r}
B(\x,f(\x)) =
%\bbm  1 & x_1 & x_1x_2 \\ x_1 &  1  & x_2  \\ x_1x_2  &  x_2  & 1 \ebm =
V_0(\x)^TV_0(\x) +
 V_1(\x)^TA(\x,\y)V_1(\x) +   V_2(\x)^TA(\x,\y)V_2(\x),
\earay
\]
where the matrix polynomials $V_i(\x)$ are:
\[
\begin{split}
V_0(\x) &= \bbm  \x_1 & 1  & \x_2  \ebm,
\\
V_1(\x) &=\bbm 1 &  -\x_1 & 0 & 0 & 0 & 1 \ebm^T
\bbm 1 & 0 &  0 \ebm,
\\
V_2(\x)& =\bbm 0  & 0  & 1 &  -\x_2 & \frac{1}{\sqrt{2}} & \frac{-1}{\sqrt{2}} \ebm^T
\bbm 0 & 0 &  1 \ebm.
\end{split}
\]
\end{exm}

\begin{exm}
[\protect{\cite[Example 4.6.3]{KSmor}}]   \label{exm:4.4}
In this example, we show that the polynomials
$V_j$ in the Positivstellensatz
certificate \eqref{Bx:fxz:QA2} 
%
%in general (this is not true? the examples do not support this)
%
might depend on $\y$.
This is the case even if there is no lifting variable $\z$.
Consider ($n=1$)
\[
A(\x,\y) := \bbm 0 & \x & 0\\ \x & \y_1 & \y_2\\
0 &\y_2 & \x
\ebm.
\]
%Then $S_A=\{0\}\times[0,\infty)\times\{0\}$ and
Clearly, $P_A=\{0\}$.
We claim that $\QM_1(A)\cap\R\cx$  is archimedean.
Obviously, $e_3^TAe_3=\x \in \QM_1(A)$. Further,
\be\label{eq:x^2}
-\x^2=uAu^T\in\QM_1(A)
\ee
for $u=\bbm \frac12+\frac12\y_1 & -\x & 0\ebm$.
Hence for each $\la>0$,
\[
1-\la \x = \left( 1-\frac{\la}2\x \right)^2-\la^2 \x^2 \in \QM_1(A).
\]
In particular, the assumptions of Theorem \ref{thm:dropDrop}
or Corollary \ref{cor:dropDrop1} are met.
However, a certificate of the form
\[
-\x^2= \sum_{i} V_{0i}^T V_{0i}+\sum_j
V_j^T A(\x, \y)  V_j \in  \QM_1(A)
\]
cannot exist for $V_{0i}, V_j \in \R\cx^3$. Indeed, if
$u= \bbm u_1&u_2&u_3 \ebm^T \in \R\cx^3$, then
\be\label{eq:uAu}
u^TAu =  2 u_1 u_2 \x + u_3^2 \x + u_2^2 \y_1
+2 u_2 u_3 \y_2.
\ee
%The sum of such terms and sums of squares
%can only eliminate $\y_i$ if all $u_2=0$.
In a sum of terms of the form \eqref{eq:uAu}, one can eliminate $\y_i$ only if all $u_2=0$.
%But then assuming this sum equals $-1$ and plugging in $\x_1=1$
%leads to a contradiction $-1\geq0.$
But, for $u_2=0$, plugging in $\x=1$
leads to the contradiction $-1\geq0$.
\end{exm}

\subsection{Matrix cube problem for spectrahedrops}
We conclude this section with an application
to the matrix cube problem.
As explained by Ben-Tal and Nemirovski~\cite{BTN,Nem},
an important problem in convex geometry and optimization is
to find the largest cube that is contained in a spectrahedrop.
Consider the linear pencil:
\be\label{eq:B}
B(\x,\z):= B_0 + \x_1 B_1 + \cdots + \x_n B_n
+ \z_1 B_{n+1} + \cdots + \z_s B_{n+s}.
\ee
The matrix cube problem is  the optimization problem
\be \label{max:rho:cube}
\max \quad \rho \quad
\mbox{s.t.} \quad  [-\rho, \rho]^n \subseteq P_B.
\ee
When $0$ is in the interior of $P_B$,
we can generally assume $B_0\succeq 0$.
Note that $[-\rho, \rho]^n \subseteq P_B$ if and only if
$[-1, 1]^n \subseteq P_{\widetilde{B}}$ with
\[
\widetilde{B} \, := \,
\frac{1}{\rho}  B_0 + \x_1 B_1 + \cdots + \x_n B_n
+ \z_1 B_{n+1}1 + \cdots + \z_s B_{n+s}.
\]
Thus, \reff{max:rho:cube} is in turn equivalent to
\be \label{min:t:PoszV(x)}
\left\{ \baray{rl}
\min &  \gm \\
\mbox{s.t.} &   \gm B_0 + {\sum}_{i=1}^n \x_i B_i +
{\sum}_{j=1}^s p_j(\x) B_{n+j} \in \QM_t(D), \\
& \gm \geq 0,
\earay \right.
\ee
for scalar polynomials $p_j(\x)$. In the above, $D(\x)$ is the diagonal matrix
\[
D(\x) = \diag( [1+\x_1 \quad 1-\x_1 \quad \cdots \quad 1+\x_n \quad 1-\x_n ] ).
\]
One can solve \reff{min:t:PoszV(x)}
as a semidefinite program, when the degrees of
the $p_j$ are chosen and a truncation of
$\QM_t(D)$ is used.

\begin{exm}
Consider the spectrahedrop $P_B$ given by the linear pencil
\[
B(\x,\z) = \bbm
   1 & \x_1 & \z_1  &  \z_3 \\
\x_1  &  1  & \x_2 &   \z_2  \\
\z_1  &  \x_2  & 1  & \x_3 \\
\z_3  &  \z_2  &  \x_3  & 1
\ebm.
\]
We want to find the largest square contained in $P_B$,
with a certificate for the inclusion.
The positive semidefiniteness of $B(x,z)$
implies that $|x_1|, |x_2|, |x_3| \leq 1$,
so $P_B$ is contained in the unit cube $[-1,1]^3$.
By solving the optimization problem \reff{min:t:PoszV(x)},
we certify that $[-1,1]^3$ is also the largest cube
contained in $P_B$. The optimal value of $\gm$
in \reff{min:t:PoszV(x)} is $1$.
The optimal $p_j$ are given as
\[
p_1 = \x_1\x_2, \quad p_2 = \x_2\x_3, \quad p_3 = \x_1\x_2\x_3.
\]
The certificate for the inclusion $P_B \subseteq [-1,1]^3$ is then
\[
B(\x,p(\x)) = V_0(\x)^TV_0(\x) + \sum_{k=1}^6 V_k(\x)^T D(\x) V_k(\x),
\]
where the $V_i(\x)$ are
\[
\begin{split}
V_0(\x) &= \bbm  1  & \x_1  & \x_1\x_2 & \x_1\x_2\x_3 \ebm,
\\
V_1(\x) &= \bbm 1 & 0 & 0 & 0 & 0 & 0 \ebm^T
\left( \frac{1-\x_1}{\sqrt{2}} \right)
\bbm 0 & 1 & \x_2 & \x_2\x_3 \ebm,
\\
V_2(\x) &= \bbm 0 & 1 & 0 & 0 & 0 & 0 \ebm^T
\left( \frac{1+\x_1}{\sqrt{2}} \right)
\bbm 0 & 1 & \x_2 & \x_2\x_3 \ebm,
\\
V_3(\x) &= \bbm 0 & 0 & 1 & 0 & 0 & 0 \ebm^T
\left( \frac{1-\x_2}{\sqrt{2}} \right)
\bbm 0 & 0 & 1 & \x_3 \ebm,
\end{split}
\]
\[
\begin{split}
V_4(\x) &= \bbm 0 & 0 & 0 & 1 & 0 & 0 \ebm^T
\left( \frac{1+\x_2}{\sqrt{2}} \right)
\bbm 0 & 0 & 1 & \x_3 \ebm,
\\
V_5(\x) &= \bbm 0 & 0 & 0 & 0 & 1 & 0 \ebm^T
\left( \frac{1-\x_3}{\sqrt{2}} \right)
\bbm 0 & 0 & 0 &  1\ebm,
\\
V_6(\x) &= \bbm 0 & 0 & 0 & 0 & 0 & 1 \ebm^T
\left( \frac{1+\x_3}{\sqrt{2}} \right)
\bbm 0 & 0 & 0 & 1 \ebm.
\end{split}
\]
We deduce that $P_B=[-1,1]^3$.
%%%%%%%%%%%%%%%%%%%%%%%%%%%%%%%%
\iffalse

We have used the following representations:
\[
B(\x,p_1,p_2,p_3) =
\bbm  1  \\ \x_1  \\ \x_1\x_2 \\ \x_1\x_2\x_3 \ebm
\bbm  1  \\ \x_1  \\ \x_1\x_2 \\ \x_1\x_2\x_3 \ebm^T +
(1-\x_1^2) \bbm 0 \\ 1 \\ \x_2 \\ \x_2\x_3 \ebm
\bbm  0 \\ 1 \\ \x_2 \\ \x_2\x_3 \ebm^T +
\]
\[
(1-\x_2^2) \bbm 0 \\  0 \\ 1 \\ \x_3 \ebm \bbm 0 \\ 0 \\ 1 \\ \x_3 \ebm^T
+ (1-\x_3^2) \bbm 0 \\ 0 \\ 0 \\ 1 \ebm \bbm 0 \\ 0 \\ 0 \\ 1 \ebm^T,
\]
\[
1-\x_i^2 = (1+\x_i)\left( \frac{1-\x_i}{\sqrt{2}} \right)^2+
(1-\x_i)\left( \frac{1+\x_i}{\sqrt{2}} \right)^2, \,
i=1,2,3.
\]

clear all, clc,
sdpvar x1 x2 x3
szA=3; szB=4;
monx = monolist([x1 x2 x3], 2);
VtV0 = sdpvar(szA*szB*length(monx));
VtV1 = sdpvar(szA*szB*length(monx));
VtVx0 = ( kron( eye(szA*szB), monx) )'*VtV0*( kron( eye(szA*szB), monx) );
VtVx1 = ( kron( eye(szA*szB), monx) )'*VtV1*( kron( eye(szA*szB), monx) );
monx2 = monolist([x1 x2 x3], 3);
p1 = sdpvar(1, nchoosek(3+3,3) );
p2 = sdpvar(1, nchoosek(3+3,3) );
p3 = sdpvar(1, nchoosek(3+3,3) );
t = sdpvar(1);
Ax = diag([1-x1^2  1-x2^2   1-x3^2]);
VtA0V = TrX( kron(eye(szB),eye(szA))*VtVx0, 2, [szB szA]);
VtA1V = TrX( kron(eye(szB),Ax)*VtVx1, 2, [szB szA]);
Bxy = [t  x1 p1*monx2  p3*monx2; 0  t  x2  p2*monx2; 0 0 t x3; 0 0 0 t];
Peq = Bxy-VtA0V-VtA1V;
coefPeq = coefficients(Peq(find(triu(Peq))),[x1 x2 x3]);
optimize([t>=0, VtV0>=0,VtV1>=0, coefPeq==0], t );
double(t),

\fi
%%%%%%%%%%%%%%%%%%%%%%%%%%%%%%%%%
\end{exm}

\section{Conclusions and discussion}
\label{sc:con}

In this paper, we have proposed a new matrix Positivstellensatz
that uses lifting polynomials.
It serves as a certificate for containment
between projections of two sets defined by polynomial matrix inequalities.
The main feature is that the lifting variables
can be parameterized by polynomials.
Such polynomials are called lifting polynomials.
A typical application of this lifted Positivstellensatz
is to certify that a spectrahedrop (i.e., projection of a spectrahedron)
is contained in another spectrahedrop.
Under some mild natural assumptions, we have shown that
the proposed lifted matrix Positivstellensatz
is a sufficient and necessary certificate for the containment.
The %Positivstellensatz 
certificate can be searched for
by solving a semidefinite program.

\subsection{The case of scalar polynomials}

Theorems~\ref{thm:dropSatz} and \ref{thm:Pos:Qsoscvz}
also apply to projections of semialgebraic sets defined by
scalar polynomials. We thus obtain a large class of
Positivstellens\"atze for projections of semialgebraic sets.
% which are not basic closed.

%
%Suppose $G(\x,\y)$, $Q(\x,\z)$ are diagonal, i.e.,
%\begin{align*}
%G(\x,\y) &= \diag \big( g_1(\x,\y), \ldots, g_k(\x,\y) \big), \\
%Q(\x,\z) &=\diag \big( q_1(\x,\z), \ldots, q_t(\x, \z) \big),
%\end{align*}
%with scalar polynomials $g_i,q_j$.
%
%For this case, we have
%\[
%\begin{split}
%P_G & = \{ x\in \re^n: \exists y\in\R^r,\, g_1(x,y) \geq 0,
%\ldots, g_k(x,y) \geq 0 \},
%\\
%P_Q &= \{ x\in \re^n: \exists z\in\R^s,\, q_1(x,z) \geq 0,
%\ldots, q_k(x,z) \geq 0 \}.
%\end{split}
%\]
Let $g_1(\x,\y), \ldots, g_k(\x,\y)$
and $q_1(\x,\z), \ldots, q_t(\x, \z)$
be scalar polynomials. They give semialgebraic sets
\be\label{eq:K1K2}
\begin{split}
K_1 & = \{ x\in \re^n: \exists y\in\R^r,\, g_1(x,y) \geq 0,
\ldots, g_k(x,y) \geq 0 \},
\\
K_2 &= \{ x\in \re^n: \exists z\in\R^s,\, q_1(x,z) \geq 0,
\ldots, q_t(x,z) \geq 0 \}.
\end{split}
\ee
We can get a Positivstellensatz certificate
for the containment $K_1 \subseteq K_2$.

\begin{cor}  \label{scalar:Posatz}
Let $g_1, \ldots, g_k \in \R[\x,\y]$
and $q_1, \ldots, q_t \in \R[\x, \z]$
be scalar polynomials,
and let $K_1,K_2$ be as in \eqref{eq:K1K2}.
Assume the quadratic module of
$(g_1,\ldots, g_k)$ is archimedean
and the degrees of $q_j$ in $z$ are at most $2d$.
If for every $x \in K_1$, each $q_j(x,\z)$ is sos-concave in $\z$
and there exists $z$ such that $q_j(x,z) > 0$,
then there exist polynomials
$ p_\af \in \R[\x]$ ($\af \in \N^s_{2d}$)
and sos polynomials $\sig_{ij} \in \R[\x,\y]$
such that $M(p)$ is an sos matrix polynomial and
for each $j=1,\ldots,t$,
\be \label{qj:sos:lift}
q_j(\x,p(\x))
 =
 \sig_{j0}(\x,\y)
 +
 \sum_{i=1}^k g_i(\x,\y) \sig_{ij}(\x,\y) .
\ee
\end{cor}
\begin{proof}
By Theorem~\ref{thm:Pos:Qsoscvz},
%the certificate~\eqref{pafP(x):in:G} holds, so
there exists a polynomial tuple $p=(p_1,\ldots,p_s)\in\R\cx^s$
and matrix polynomials $V_i(\x,\y)$ such that
\[
\begin{split}
 \diag \big( q_1(\x, p(\x) ), \ldots, q_t(\x, p(\x) ) \big) &= \\
 \sum_i  V_i(\x,\y)^T \diag \big( g_1(\x,\y), \ldots, g_k(\x,\y) \big)  V_i(\x,\y)
&+ V_0(\x,\y)^T V_0(\x,\y).
%= & \sum_{j=1}^k g_j(\x,\y) \sum_i
%V_{ij}(\x,\y)^T V_{ij}(\x,\y) + V_0(\x,\y)^T V_0(\x,\y).
\end{split}
\]
Comparing diagonal entries, we  see that \reff{qj:sos:lift}
holds for some sos polynomials $\sig_{ij}(\x,\y)$.
\end{proof}

\subsection{Some open questions}

In future research, the following interesting and important  questions
should be addressed.
They are mostly open to the authors.

\begin{ques}
In the certificates \reff{B(x,y):Postz:A(x,z)},
\reff{Pp(x):in:Q(G)}, or \reff{pafP(x):in:G},
for what kinds of matrix polynomials $G(\x,\y)$ and $Q(\x,\z)$,
can we choose the polynomials $V_j$ to be independent of $\y$?
\end{ques}

The above question is of great interest in computation.
If each $V_j$ is independent of $\y$, the semidefinite programs
searching for \reff{B(x,y):Postz:A(x,z)},
\reff{Pp(x):in:Q(G)}, or \reff{pafP(x):in:G}
become much easier to solve.
In Example~\ref{exm:4.4}, the polynomials $V_j$ must depend on $\y$.
However, in all the other examples,
we can choose $V_j$ to be independent of $\y$.

Convexity is used in a key step in the proofs 
of Theorems \ref{thm:dropSatz} and \ref{thm:Pos:Qsoscvz} to obtain a 
lifting polynomial
map $P_G\to S_Q$. 
 When $Q(\x,\z)$ is not convex in $\z$, the lifting polynomials might not exist, 
 cf.~Example \ref{ex:nolift}. 
 This leads to the following problem:

\begin{ques}
In Theorem~\ref{thm:Pos:Qsoscvz},
when $Q(\x,\z)$ is not sos-concave in $\z$,
what is an appropriate certificate for ensuring
$P_G \subseteq P_Q$?
\end{ques}

This question is generally very challenging.
Indeed,
for a given $x$, checking the existence of
$z$ satisfying $Q(x,z) \succeq 0$ is already difficult.
This requires solving a polynomial matrix inequality,
which is computationally very demanding.

\begin{ques}
For two linear pencils $A(\x,\y)$ and $B(\x,\z)$,
what is the appropriate certificate for $P_A = P_B$?
\end{ques}

The certificate \reff{Bx:fxz:QA2} ensures $P_A \subseteq P_B$.
To ensure $P_B \subseteq P_A$, one might
be tempted to apply a similar certificate again.
However, this usually does not work
because \reff{Bx:fxz:QA2} typically requires $P_A \subseteq  \mbox{int}(P_B)$.
To get a similar certificate for $P_B \subseteq P_A$,
one usually needs $P_B \subseteq  \mbox{int}(P_A)$.
Clearly, $P_A \subseteq  \mbox{int}(P_B)$
and $P_B \subseteq  \mbox{int}(P_A)$
generally do not hold simultaneously.

\linespread{.996}

\subsection*{Acknowledgments}
Igor Klep was partially supported by the Marsden Fund Council
of the Royal Society of New Zealand
and partially supported by
the Slovenian Research Agency grants  J1-8132, N1-0057 and P1-0222.
He would like to thank Christoph Hanselka for helpful discussions.
Jiawang Nie was partially supported by the NSF grants
DMS-1417985 and DMS-1619973.
The research was initiated when Jiawang Nie
visited the University of Auckland,
under the support of a Kalman visiting Fellowship.

\end{document}